\documentclass[11pt]{amsart}

\usepackage{comment}
\usepackage{color}
\usepackage{amsmath}
\usepackage{amsthm}
\usepackage{amssymb}
\usepackage{graphicx}
\usepackage{enumerate}
\usepackage{amsfonts}
\usepackage{stackrel}
\usepackage{mathrsfs}
\usepackage{parskip}
\usepackage{mathdots}
\usepackage{color}

\numberwithin{equation}{section}
\theoremstyle{plain}
\newtheorem{Proposition}[equation]{Proposition}
\newtheorem{Corollary}[equation]{Corollary}
\newtheorem*{Corollary*}{Corollary}
\newtheorem{Theorem}[equation]{Theorem}
\newtheorem*{Theorem*}{Theorem}
\newtheorem{Lemma}[equation]{Lemma}
\theoremstyle{definition}

\newtheorem{Example}[equation]{Example}
\newtheorem{Remark}[equation]{Remark}

\allowdisplaybreaks

\def\D{\mathbb{D}}
\def\T{\mathbb{T}}

\def\K{\mathcal{K}}
\def\phi{\varphi}

\newcommand{\beqa}{\begin{eqnarray*}}
\newcommand{\eeqa}{\end{eqnarray*}}

\renewcommand{\geq}{\geqslant}

\renewcommand{\subset}{\subseteq}

\title[Dual of the compressed shift]{The dual of the compressed shift}

\author[Camara]{M. C. C\^{a}mara}
\address{Departamento de Matematica, Instituto Superior Tecnico, 1049-001 Lisboa, Portugal}
\email{cristina.camara@tecnico.ulisboa.pt}

\author[Ross]{W. T. Ross}
	\address{Department of Mathematics and Computer Science, University of Richmond, Richmond, VA 23173, USA}
	\email{wross@richmond.edu}
	


\thanks{This work was partially supported by FTC/Portugal through the grant UID/MAT/04459/2019. The second author would like to thank the Center for Mathematical Analysis, Geometry, and Dynamical Systems for their hospitality where the initial research for this paper was done.}

\begin{document}

\begin{abstract}
For an inner function $u$ we discuss the dual operator for the compressed shift $P_u S|_{\K_u}$, where $\K_u$ is the model space for $u$. We describe the unitary equivalence/similarity classes for these duals as well as their invariant subspaces.
\end{abstract}

\maketitle

\section{Introduction}

This paper deals with the unitary equivalence classes and the invariant subspaces of the dual operators for the well-known compressed shift operator on a model space. The main tool to explore these results is to connect these dual operators to the bilateral shift  on $L^2$ as well as a direct sum of the unilateral shift and its adjoint.

For an inner function $u$ on  $\D := \{|z| < 1\}$, consider the {\em model space} \cite{MR3526203}
$$\K_{u} := H^{2} \cap (u H^2)^{\perp},$$
where $H^2$ is the Hardy space \cite{Duren}. By Beurling's theorem, the subspaces $u H^2$ are the non-zero invariant subspaces of the shift operator 
\begin{equation}\label{shift}
(S f)(z) = z f(z)
\end{equation} on $H^2$ and thus, via annihilators, the spaces $\K_u$ are the non-trivial $S^{*}$-invariant subspaces of $H^2$. The operator $S^{*}$ can  be realized as the backward shift 
\begin{equation}\label{bkshirt88899}
(S^{*} f)(z) = \frac{f(z) - f(0)}{z}.
\end{equation} 

As $H^2$ is a closed subspace of $L^2(\T, d \theta/2\pi)$, one denotes $P_{u}$ to be the orthogonal projection of $L^2$ onto $\K_u$. The operator 
$$S_{u}:= P_{u} S|_{\K_{u}},$$ is called the {\em compressed shift} and plays an important role in operator theory \cite[p.~195]{MR3526203}. 

Related to $S_{u}$ are the {\em truncated Toeplitz operators} 
$A^{u}_{\phi} := P_{u} M_{\phi}|_{\K_{u}},$
where $\phi \in L^{\infty}$ and 
$M_{\phi} f= \phi f$ is multiplication by $\phi$ on $L^2$. Note that $A_{z}^{u} = S_{u}$. These truncated Toeplitz operators  have received considerable attention since their initial introduction in \cite{MR2363975} (see also \cite{MR3589670, MR3052299}). 

The recent papers \cite{MCKB, MR3759573, BHal, Mbbhbbh} began an interesting study of  the {\em dual truncated Toeplitz operators} $D^{u}_{\phi}$, $\phi \in L^{\infty}$, defined on 
$\K_{u}^{\perp}$
by 
$$
D_{\phi}^{u} := (I - P_{u}) M_{\phi}|_{\K_{u}^{\perp}}.
$$
Notice that $I - P_u$ is the orthogonal projection of $L^2$ onto $\K_{u}^{\perp}$. 
Decomposing $L^2$ as 
$L^{2} = \K_{u} \oplus \K_{u}^{\perp},$ one can think of  $A_{\phi}^{u}$ and its associated dual $D_{\phi}^{u}$ as parts of the multiplication operator 
$$M_{\phi}: L^2 = \K_{u} \oplus \K_{u}^{\perp} \to L^2, \quad M_{\phi} f = \phi \cdot f,$$ by means of its matrix decomposition 
\begin{equation}\label{6yhn5tgbwedc}
M_{\phi} = 
\begin{bmatrix}
A^{u}_{\phi} & \ast\\
\ast & D^{u}_{\phi}
\end{bmatrix}.
\end{equation}

In this paper, we focus on the {\em dual of the compressed shift} $S_{u}$, denoted by 
\begin{equation}\label{dedefineirne}
D_{u} := (I - P_{u}) S|_{\K_{u}^{\perp}}.
\end{equation}
By \eqref{6yhn5tgbwedc}, we can understand $D_u$ in terms of matrices as 
$$M := \begin{bmatrix}
S_{u} & \ast\\
\ast & D_{u}
\end{bmatrix},$$
where $M := M_z$ on $L^2$ and the matrix above is with respect to the orthogonal decomposition $L^2 = \K_u \oplus \K_{u}^{\perp}$. There are  other contexts of dual operators defined for  Toeplitz and subnormal operators \cite{MR3313402, MR3608059, MR617974, MR1885661} and thus these duals enjoy a tradition in operator theory. 

Along with a discussion of some basic properties of $D_u$, we will describe the $D_{u}$ invariant subspaces of $\K_{u}^{\perp}$ as well as the similarity and unitary equivalence properties of $D_u$ and $D_v$ for inner $u$ and $v$. We will show that when $u(0) = 0$, $D_u$ is unitarily equivalent to $S \oplus S^{*}$ on $H^2 \oplus H^2$, and thus $D_u$ and $D_v$ are unitarily equivalent whenever $u(0) = v(0) = 0$. When  $u(0) \not = 0$, $D_u$ turns out to be similar to $M$ on $L^2$, and thus $D_u$ is similar to $D_v$ whenever $u(0) \not = 0$, $v(0) \not = 0$.  Finally, we show that $D_{u}$ is unitarily equivalent to $D_v$ precisely when $|u(0)| = |v(0)|$. These results become important when describing  the invariant subspaces of $D_u$ (sections \ref{unotzero} and \ref{uzero}) and have connections to results from \cite{MR3816130} and \cite{CHI}. 
After this paper was completed, we learned of the paper \cite{CHI} which approaches the $D_{u}$-invariant subspaces of $\K_{u}^{\perp}$ in a different way. 
% as well as a description of $\sigma(D_u)$,  the spectrum of $D_u$. We will show that $\sigma(D_u) = \overline{\D}$ when $u(0) = 0$ while $\sigma(\D_u) = \T$ when $u(0) \not = 0$. 
\section{Some basics}

The space $L^{2} = L^{2}(\T, dm)$, where $\T$ is the unit circle and $m = d\theta/2\pi$ on $\T$, is a Hilbert space with inner product 
$\langle f, g\rangle := \langle f, g\rangle_{L^2}$. The Fourier coefficients of $f$ will be denoted by $\widehat{f}(j) = \langle f, \xi^{j}\rangle$. 
%Moreover, the set $\{z^{j}: j \in \Z\}$ is an orthonormal basis for $L^2$ and Parseval's theorem says that if $f, g \in L^2$ with Fourier series
%$$f \sim \sum_{j \in \Z} \widehat{f}(j) z^{j}, \quad g \sim \sum_{j \in \Z} \widehat{g}(j) z^{j},$$ 
%where 
%$$\widehat{f}(j) :=  \int_{\T}  f(z) \overline{z^{j}} dm(z), \quad j \in \Z,$$
%are the Fourier coefficients of $f$, 
%then 
%$$\langle f, g\rangle = \sum_{j \in \Z} \widehat{f}(j) \overline{\widehat{g}(j)}.$$
Viewing the Hardy space $H^2$ as 
$\{f \in L^2: \widehat{f}(n)  = 0 \;  \forall n < 0\}$
and 
$\overline{H^{2}_{0}}$ as $\{\overline{z f}: f \in H^2\},$
note that $L^2 = H^2 \oplus \overline{H^{2}_{0}}$.
%$$(H^2)^{\perp} = \overline{H^{2}_{0}}, \quad (\overline{H^{2}_{0}})^{\perp} = H^2.$$
Let $P_{_+}$ and $P_{-}$ denote the standard orthogonal projections from $L^2$ onto $H^2$ and $\overline{H_{0}^{2}}$ respectively.

For an inner function $u$, define the model space 
$\K_{u} = H^{2} \cap (u H^2)^{\perp}.$
Elementary facts about annihilators  will verify that 
$$\K_{u}^{\perp} = \overline{H^{2}_{0}} \oplus u H^2.$$
As $\K_{u}$ is a closed subspace of $L^2$, we have an orthogonal projection $P_{u}$ from $L^2$ onto $\K_{u}$. 
A result from  \cite[p.~124]{MR3526203} relates $P_u, I - P_u, P^{+}$, and $P^{-}$.

\begin{Lemma}\label{03r98egiodfsja}
If $u$ is inner, then 
$P_{u} = P^{+} - M_u P^{+} M_{\overline{u}} = M_{u} P^{-} M_{\overline{u}} P^{+}$
and 
$I - P_{u} = P^{-} + M_{u} P^{+} M_{\overline{u}}.$
\end{Lemma}

Any $f \in L^2 = H^{2} \oplus \overline{H^{2}_0}$ can be written uniquely as 
$$f = f_{+} + f_{-}, \quad f_{+} \in H^{2}, \quad f_{-} \in \overline{H^{2}_{0}},$$
that is, $f_{+} = P^{+} f$ and $f_{-} = P^{-} f$.

%Since $f_{+}$ can be regarded as an analytic function on $\D$,  we can express $P^{+}$ as an integral operator 
%\begin{equation}\label{ppPpppppPP}
%(P^{+} g)(z) = \int_{\T} \frac{1}{1 - \overline{\xi} z} g(\xi) \,dm(\xi), \quad g \in L^2, \quad z \in \D.
%\end{equation}
%In a similar way, $f_{-}$ can be regarded as a co-analytic function on $\D$ that vanishes at the origin and we can express $P^{-}$ as the integral operator 
%\begin{equation}\label{mMMmmMMmmM}
%(P^{-} g)(z) = \int_{\T} \frac{\xi \overline{z}}{1 - \xi \overline{z}} g(\xi) \,dm(\xi), \quad g \in L^2, \quad z \in \D.
%\end{equation}

 We will also use the notation 
\begin{equation}\label{zZxXvV}
\phi_f := \overline{z} \overline{f_{-}}, \quad f \in L^2.
\end{equation}
Observe that $\phi_{f} \in H^2$, and hence is analytic on $\D$, and so we can utilize the quantity $\phi_{f}(0)$. A Fourier series argument will show that 
\begin{equation}\label{ooopopoddGt666}
\phi_{f}(0) = \int_{\T} \overline{z}\overline{ f_{-}} \,dm =\overline{ \widehat{f_{-}}(-1)}.
\end{equation}

Any $f \in  \K_{u}^{\perp} = \overline{H^{2}_{0}} \oplus u H^2$ can be written uniquely as 
\begin{equation}\label{7765432qq}
f = f_{-} + u \widetilde{f_{+}}, \quad f_{-} \in \overline{H^{2}_{0}}, \quad \widetilde{f_{+}} \in H^2.
\end{equation}
 Lemma \ref{03r98egiodfsja} shows that 
$f_{-} = P^{-} f$ and $\widetilde{f_{+}} = P^{+}(\overline{u} f)$ and a Fourier series argument will verify the following identities. 

\begin{Lemma}\label{uUyyHHt}
For $f \in L^2$ we have
\begin{enumerate}
\item[(i)] $P^{-}(z f_{-}) = \overline{\phi_{f}} - \overline{\phi_{f}(0)}$;
\item[(ii)] $P^{-}(\overline{z} f_{+}) = f_{+}(0) \overline{z}$;
\item[(iii)] $P^{+}(\overline{z} f_{+}) = (f_{+} - f_{+}(0)) \overline{z}$;
\item[(iv)] $P^{+}(z f_{-}) = \overline{\phi_{f}(0)}$.
\end{enumerate}
\end{Lemma}

%In this paper all of our functions we will be regarded as $L^2$ functions, even though some of them will have extensions to analytic functions on $\D$. In this regard, we have the following technical result. 

Regarding $\K_{u}$ as a subspace of $L^2$, we have the following useful result. 

\begin{Proposition}\label{978uwioejfskdlac}
If $u$ is inner then $\overline{u} \K_{u} = \overline{z} \overline{\K_{u}}$.
\end{Proposition}

\begin{proof}
It is a standard fact \cite{MR3526203} that the conjugate-linear operator 
\begin{equation}\label{conjugationoperator}
C_{u}: L^2 \to L^2, \quad C_{u} f = u \overline{z f},
\end{equation}
 is an involutive isometry on $L^2$ with 
$
C_{u} \K_u  = \K_u$ and $C_{u} \K_{u}^{\perp} = \K_{u}^{\perp}. 
$
Thus $\overline{u} \K_u  = \overline{u} C_{u} \K_{u}
= \overline{z} \overline{\K_{u}}$.
\end{proof}

The model space $\K_{u}$ is a reproducing kernel Hilbert space on $\D$ with kernel 
$$
k^{u}_{\lambda}(z) = \frac{1 - \overline{u(\lambda)} u(z)}{1 - \overline{\lambda} z}, \quad \lambda, z \in \D,
$$
meaning that 
$f(\lambda) = \langle f,  k^{u}_{\lambda}\rangle$ for $f \in \K_{u}$ and $\lambda \in \D$ \cite[p.~111]{MR3526203}.

\section{Some basic facts about the dual}

In this section we will develop some useful facts about $D_u$. We start with a more useful formula for $D_{u}$ than the one in  \eqref{dedefineirne}.
% We make note of the function 
%$k_{0}^{u}(z) = 1 - \overline{u(0)} u(z)$
%from \eqref{uUyyHHt} which will appear many times throughout this paper.
% In the following result, recall the definition of $\phi_{f}$ from \eqref{zZxXvV}.

\begin{Proposition}\label{0982iu4hrfjdslkva}
If $u$ is inner then 
$$D_{u} f = z f - \overline{\phi_{f}(0)} k^{u}_{0}, \quad f \in \K_{u}^{\perp}.$$
\end{Proposition}

\begin{proof}
For $f = f_{-} + u \widetilde{f_{+}} \in \K_{u}^{\perp}$ use Lemma \ref{03r98egiodfsja} to see that 
\begin{align*}
D_{u} f & = (I - P_{u}) (z f) = (P^{-} + u P^{+} \overline{u})(z f_{-} + z u \widetilde{f_{+}})\\
& = z f_{-} - \overline{\phi_{f}(0)} + u P^{+} z \overline{u} f_{-} + z u \widetilde{f_{+}}\\
& = z f_{-} + z u \widetilde{f_{+}} - \overline{\phi_{f}(0)} + u \overline{u(0)} \overline{\phi_{f}(0)}\\
& = z f - \overline{\phi_{f}(0)} (1 - \overline{u(0)} u)
 = z f - \overline{\phi_{f}(0)} k_{0}^{u}.
\end{align*}
Note the use of Lemma \ref{uUyyHHt} and $\overline{u} f_{-} = \overline{z} \overline{u} \overline{\phi_{f}}$ and $\phi_{\overline{u} f_{-}} = u \phi_{f}.$
%\begin{align*}
%D_{u} f & = (I - P_{u}) (z f)\\
%& = z f - P_{u} (z f)\\
%& = z f - P_{u}(z f_{-} + z u \widetilde{f_{+}})\\
%& = z f - P_{u} (z f_{-})- P_{u}(z u \widetilde{f_{+}})\\
%& = z f - \underbrace{u P^{-} \overline{u} P^{+}(z f_{-}) - u P^{-}(z \widetilde{f_{+}})}_{\mbox{\tiny Lemma \ref{03r98egiodfsja}}}\\
%& = z f - \underbrace{u P^{-}(\overline{u} \overline{\phi_{f}(0)})}_{\mbox{\tiny Lemma \ref{uUyyHHt}}}- 0\\
%& = z f - \overline{\phi_{f}(0)} u P^{-}(\overline{u})\\
%& = z f - \overline{\phi_{f}(0)} u P^{-}((\overline{u}  - \overline{u(0)}) + \overline{u(0)})\\
%& = z f - \overline{\phi_{f}(0)} u (\overline{u} - \overline{u(0)})\\
%& = z f - \overline{\phi_{f}(0)} k^{u}_{0}. 
%\end{align*}
%Using \eqref{ooopopoddGt666} yields the second identity of the proposition. 
\end{proof}

\begin{comment}
\begin{Corollary}
For all inner functions $u$ for which $u(0) = 0$ we have 
$$D_{u} f = z f - \int_{\T} z f dm = z f - \overline{\phi_{f}(0)}, \quad f \in K_{u}^{\perp}.$$ 
\end{Corollary}
\end{comment}

\begin{Corollary}\label{1998764}
If $u$ is inner then
$D_{u}|_{u H^2} = S|_{u H^2}$ and thus 
$D_{u}(u H^2) \subset u H^2$. When $u(0) = 0$, we have $D_{u} \overline{H^{2}_{0}} = \overline{H^{2}_{0}}$.
\end{Corollary}

The definition of $D_{u}$ from \eqref{dedefineirne} shows that 
$
D_{u}^{*} = D_{\overline{z}}^{u}.
$
 In fact, via the conjugation operator $C_u$ from \eqref{conjugationoperator}, we have 
$
C_{u} D_{u} C_{u} = D_{u}^{*}
$ \cite{MR3759573}.
Proposition \ref{0982iu4hrfjdslkva} and the above conjugation identity yield the following. 

\begin{Proposition}
If $u$ is inner then $D_{u}^{*} f = \overline{z} f - \widetilde{f_{+}}(0) C_{u} k_{0}^{u}$. Furthermore, $D_{u}^{*}|_{\overline{H^{2}_{0}}} = M^{*}|_{\overline{H^{2}_{0}}},$ where $M^{*} f = \overline{z} f$, and thus $D_{u}^{*} \overline{H^{2}_{0}} \subset \overline{H^{2}_{0}}$. When $u(0) = 0$, we have $D_{u}^{*} (u H^2) = u H^2$. 
\end{Proposition}

Here are some interesting facts from  \cite{Group} and  \cite{MR3759573} about $D_u$.

\begin{Proposition}\label{uuuuhhww}
For an inner function $u$ we have the following:
\begin{enumerate}
\item[(i)] $\|D_{u}\| = 1.$
%\item[(ii)] $C_{u} D_{u} C_{u}  = D_{u}^{*},$
\item[(ii)]  $\sigma(D_u) = \overline{\D}$ when $u(0) = 0$ while $\sigma(D_u) = \T$ when $u(0) \not = 0$.
\item[(iii)] $D_{u} D_{u}^{*} = I -(1-|u(0)|^2)u\otimes u$.
%\item[(v)]  $D_{u}^{*} D_{u} =I -(1-|u(0)|^2) \bar z\otimes\bar z$.
\end{enumerate}
\end{Proposition}

\section{Unitary equivalence and similarity}

For two compressed shifts $S_u$ and $S_v$ we know that $S_{u}$ is unitarily equivalent to $S_v$ if and only if $u$ is a constant unimodular multiple of $v$. For their duals, they are often unitarily equivalent and even more often similar.  This will be an important part of our analysis of their invariant subspaces. 

For an inner function $u$, the authors in \cite{BHal} define the onto isometry 
\begin{equation}\label{UUUUuuuu}
U: L^2 = H^2 \oplus \overline{H^{2}_{0}} \to \K_{u}^{\perp} = u H^2 \oplus \overline{H^{2}_{0}}, \quad 
U = \begin{bmatrix}
M_{u} & 0\\
0 & I
\end{bmatrix},
\end{equation}
where recall that $M_{u} f = u \cdot f$ on $L^2$. 
%Note that 
%$$U^{*} = 
%\begin{bmatrix}
%M_{\overline{u}} & 0\\
%0 & I
%\end{bmatrix}.
%$$
A computation in that paper yields the following lemma.  For any $\phi \in L^{\infty}$ recall the definition of the Hankel operator 
$H_{\phi}: H^2 \to \overline{H_{0}^{2}}$, $H_{\phi}f = P^{-}(\phi f)$ as well as the following formula for its adjoint 
$H_{\phi}^{*}: \overline{H^{2}_{0}} \to H^2$, $H_{\phi}^{*} f =  P^{+}(\overline{\phi} f).$

\begin{Lemma}
For an inner function $u$ we have
$$U^{*} D_{u} U = 
\begin{bmatrix}
S & H_{u \overline{z}}^{*}\\
0  & Q
\end{bmatrix},$$
where $S$ is the shift on $H^2$ from \eqref{shift} and 
$Q: \overline{H^{2}_{0}} \to \overline{H^{2}_{0}}$, $Q g = P^{-} (z g).$
\end{Lemma}

%\begin{proof}
%The paper \cite{BHal} proves a more general result than what we prove here. For $f \in H^2$ Corollary \ref{1998764} yields
%$D_{u}(u f) = u z f.$
%If $g \in \overline{H^{2}_{0}}$ Lemma \ref{03r98egiodfsja} yields
%$$D_{u} g = u P^{+}(\overline{u} z g) + P^{-}(z g).$$
%Thus 
%\begin{align*}
%U^{*} D_u U \begin{pmatrix} f\\ g \end{pmatrix} & = U^{*} D_u \begin{pmatrix} u f\\ g\end{pmatrix}\\
%& = U^{*} \begin{pmatrix} u z f + u P^{+} (\overline{u} z g)\\ P^{-}(z g)\end{pmatrix}\\
%& = \begin{pmatrix} z f + P^{+}(\overline{u} z g)\\ P^{-}(z g)\end{pmatrix}\\
% & = \begin{pmatrix}
%S & H_{u \overline{z}}^{*}\\
%0  & Q
%\end{pmatrix} \begin{pmatrix} f\\ g\end{pmatrix}. \qedhere
%\end{align*}
%\end{proof}

One of the main theorems of this section is the following.

\begin{Theorem}\label{uu6663455}
Let $u$ be an inner function. 
\begin{enumerate}
\item[(i)] If  $u(0) = 0$, then $D_{u}$ is unitarily equivalent to the operator 
$$\begin{bmatrix}
S & 0\\
0 & Q
\end{bmatrix}: H^2 \oplus \overline{H^{2}_{0}} \to H^2 \oplus \overline{H^{2}_{0}}$$
and thus for any two inner functions $u$ and $v$ which vanish at $0$, the operators $D_u$ and $D_v$ are unitarily equivalent. 
\item[(ii)] If $u(0) \not = 0$, then $D_{u}$ is unitarily equivalent to the operator
$$\begin{bmatrix}
S & \overline{u(0)} (1 \otimes \overline{z})\\
0 & Q
\end{bmatrix}: H^2 \oplus \overline{H^{2}_{0}} \to H^2 \oplus \overline{H^{2}_{0}}.$$
\end{enumerate}
\end{Theorem}

\begin{proof}
If $u(0) = 0$, then $H^{*}_{u \overline{z}} \equiv 0$. Indeed,
for $g \in \overline{H^{2}_{0}}$,
$$H_{u \overline{z}}^{*} g = P^{+}(z \overline{u} g) = P^{+}\big( \frac{\overline{u}}{\overline{z}} \cdot g\big) = 0,$$
since $\overline{u}/\overline{z} \in \overline{H^2}$ and thus $(\overline{u}/\overline{z}) g \in \overline{H^{2}_{0}}$.

When $u(0) \not = 0$ and $g \in \overline{H^{2}_{0}}$ we can use Lemma \ref{uUyyHHt}(iv) and \eqref{ooopopoddGt666} to get 
$$
H_{u \overline{z}}^{*} g  = P^{+}(\overline{u} z g) = \overline{\phi_{\overline{u} g}(0)} = \overline{u(0)} \overline{\phi_{g}(0)} = \overline{u(0)} \widehat{g}(-1).$$
But  this is the rank one operator
$\overline{u(0)} (1 \otimes \overline{z}): \overline{H^{2}_{0}} \to H^2.$
\end{proof}

We can refine this a bit further. Recall  $S$ and  $S^{*}$  from \eqref{shift} and \eqref{bkshirt88899}.

\begin{Corollary}\label{bb77ujmkjhnb}
If $u$ is inner and $u(0) = 0$, then $D_u$ is unitarily equivalent to  $S \oplus S^{*}$ on $H^2 \oplus H^2$. 
\end{Corollary}

\begin{proof}
Via the unitary operator $U$ from \eqref{UUUUuuuu}, we see from Theorem \ref{uu6663455} that $D_u$ is unitarily equivalent to  $S \oplus Q$ on $H^2 \oplus \overline{H^{2}_{0}}$, where $Q g = P^{-}(z  g)$, $g \in \overline{H^{2}_{0}}$. One can quickly check that 
$W: \overline{H^{2}_{0}} \to H^2$, $(W g)(z) = g(\overline{z})/z$
is unitary with $S^{*} W = W Q$. Thus the unitary operator 
$L = I \oplus W: H^2 \oplus \overline{H^{2}_{0}} \to H^2 \oplus H^2$ will satisfy $(S \oplus S^{*}) L = L (S \oplus Q)$. 
\end{proof}

We will refine this unitarily equivalence result further in Theorem \ref{ooOOooiiUUP} below. 

As it turns out, all of the operators $D_u$, when $u(0) \not = 0$, are similar to the bilateral shift $M f = z f$ on $L^2$. This is important observation will come into play when discussing the invariant subspaces for $D_u$. To this end, for $u$ inner with $u(0) \not = 0$, define 
\begin{equation}\label{09iuhgiuyhgfuygf}
V_{u}: \K_{u}^{\perp} \to L^2, \quad V_{u} := P^{-} + \frac{\overline{u}}{\overline{u(0)}} P^{+}
\end{equation}
with inverse 
\begin{equation}\label{uub75chHyYYUUIO}
V_{u}^{-1}: L^2 \to \K_{u}^{\perp}, \quad V_{u}^{-1} = P^{-} + u \overline{u(0)} P^{+}.
\end{equation}
Observe that 
\begin{equation}\label{llLKKkkKK}
V_{u} = P^{-} + \frac{1}{\overline{u(0)}} P^{+} \overline{u}.
\end{equation}

\begin{Theorem}\label{uuUUuuuv78778768768}
If $u$ is inner with $u(0) \not = 0$, then $V_{u} D_{u} V_{u}^{-1} = M$ on $L^2$. Consequently, for and inner $u$ and $v$ with $u(0) \not = 0$, $v(0) \not = 0$, $D_u$ is similar to $D_v$ and $D_{u} = W^{-1} D_{v} W$, where $W: \K_{u}^{\perp} \to \K_{v}^{\perp}$, $W = P^{-} + \frac{\overline{v(0)}}{\overline{u(0)}} v P^{+} \overline{u}$.
\end{Theorem}

\begin{proof}
For $f  = f_{-} + f_{+} \in L^2$ use Proposition \ref{0982iu4hrfjdslkva} and Lemma \ref{uUyyHHt} to obtain 
\begin{align*}
& V_{u} D_{u} V_{u}^{-1} (f_{-} + f_{+}) = V_{u} D_{u} (f_{-} + \overline{u(0)} u f_{+})\\
 & = (P^{-} + \frac{\overline{u}}{\overline{u(0)}} P^{+}) (z f_{-} + z u \overline{u(0)} f_{+} - \overline{\phi_{f}(0)} + \overline{\phi_{f}(0)} \overline{u(0)} u)\\
& = \overline{\phi_{f}} - \overline{\phi_{f}(0)} + \frac{\overline{u}}{\overline{u(0)}} \overline{\phi_{f}(0)} + z f_{+} - \frac{\overline{u}}{\overline{u(0)}} \overline{\phi_{f}(0)} + \overline{\phi_{f}(0)}\\
& = z f_{-} + z f_{+} = M f.
\end{align*}
From here it follows that $D_{u} = W^{-1} D_{v} W$ with $W = V_{v}^{-1} V_{u}$. 
\end{proof}

%By using inner products, one can check the following. 

%\begin{Proposition}
%For any inner $u$ with $u(0) \not =  0$, the formula for $V^{*}: L^2 \to \K_{u}^{\perp}$ is given by 
%$V^{*} = P^{-} + \frac{u}{u(0)} P^{+}.$ Since $D_{u}$ is similar to $M$ via $V$, we also see that $D_{u}^{*}$ is similar to $M^{*}$  via $V^{*}$. 
%\end{Proposition}

\begin{Remark}
It is important to point out that although $D_u$ is similar to $M$ when $u(0) \not = 0$, it is not unitarily equivalent to $M$. This is because $M$ is normal while $D_u$ is not (Proposition \ref{uuuuhhww}). It also follows that $D_u$ is not similar to $D_v$ when $u(0) = 0$ and $v(0) \not = 0$ (Proposition \ref{uuuuhhww}).
\end{Remark}

We return to  the unitary equivalence of $D_u$ and $D_v$ begun in Theorem \ref{uu6663455}. 

\begin{Theorem}\label{ooOOooiiUUP}
If $u$ and $v$ are inner functions then $D_u$ is unitarily equivalent to $D_v$ if and only if $|u(0)| = |v(0)|$. 
\end{Theorem}

\begin{proof}
When $u(0) = v(0) = 0$, the result follows from Theorem \ref{uu6663455}. So  assume that $u(0)$ and $v(0)$ are both nonzero. Suppose $Z: \K_{u}^{\perp} \to \K_{v}^{\perp}$ is unitary with $Z D_{u} Z^{*} = D_v$. From Proposition \ref{uuuuhhww} we have 
\begin{align*}
I|_{\K_{v}^{\perp}} - (1 - |u(0)|^2) Z u \otimes Z u & = Z D_{u} D_{u}^{*} Z^{*} = D_{v} D_{v}^{*}\\ &= I|_{\K_{v}^{\perp}} - (1 - |v(0)|^2) v \otimes v
\end{align*}
and it follows that 
$$(1 - |u(0)|^2) Z u \otimes Z u = (1 - |v(0)|^2) v \otimes v.$$
Apply both sides to the unit vector $v \in v H^2 \subset \K_{v}^{\perp}$ and observe that 
$$(1 - |u(0)|^2) \langle v, Z u\rangle Z u = (1 - |v(0)|^2) v$$
implying that $Z u = c v$ for some unimodular constant $c$ (because $u$ and $v$ are unit vectors and $Z$ is unitary). The previous equation yields $|u(0)| = |v(0)|$. 

Conversely, if $|u(0)| = |v(0)|$ then Theorem \ref{uuUUuuuv78778768768} yields
$D_{u} = W^{-1} D_v W$ where $$W = P^{-} + \frac{\overline{v(0)}}{\overline{u(0)}} v P^{+} \overline{u} \quad \mbox{and} \quad W^{-1} = P^{-} + \frac{\overline{u(0)}}{\overline{v(0)}} u P^{+} \overline{v} = W^{*}$$ since $\frac{v(0)}{u(0)} = \frac{\overline{u(0)}}{\overline{v(0)}}$. Therefore $W$ is unitary. 
\end{proof}

\section{Invariant subspaces}

%For the compressed shift $S_u$, the invariant subspaces are known \cite[p.~194]{MR3526203}.

%\begin{Theorem}
%Let $u$ be any inner function. For an inner function $v$ which divides $u$, the space 
%$v H^2 \cap \K_u$ is an invariant subspace for $S_u$. Furthermore, any $S_u$-invariant subspaces takes this form. 
%\end{Theorem}

%The invariant subspaces for $D_u$ are much richer. 

We begin our discussion  with a few general results. 

\begin{Proposition}\label{jjJjjllLll}
Let $u$ be any inner function.
A subspace $\mathscr{S} \subset \K_{u}^{\perp}$ is  $D_u$-invariant with $\mathscr{S} \subset u H^2$, or equivalently $P^{-} \mathscr{S} = \{0\}$, if and only if $\mathscr{S} = \gamma u H^2$ for some inner function $\gamma$.
\end{Proposition}

\begin{proof}
If $\mathscr{S} = \gamma u H^2$, then $\mathscr{S} \subset u H^2$ \cite[p.~87]{MR3526203} and by Corollary \ref{1998764}, 
$D_{u} \mathscr{S} = z \mathscr{S} \subset \mathscr{S}$ and so $\mathscr{S}$ is $D_{u}$-invariant.  Conversely, when $\mathscr{S} \subset u H^2$ is a $D_u$-invariant, then, again by Corollary \ref{1998764}, $S \mathscr{S} \subset \mathscr{S}$. By Beurling's theorem, $\mathscr{S} = \beta H^2$ for some inner $\beta$. But $\beta H^2 \subset u H^2$ and so $\beta = \gamma u$.
\end{proof}

\begin{Lemma}\label{888UyT}
For a non-zero subspace $X \subset \overline{H^{2}_{0}}$ we have 
$X = \overline{z} \overline{\K_{\alpha}}$ for some inner $\alpha$ if and only if $P^{-}(z X) \subset X$ and $X \not = \overline{H^{2}_{0}}$. 
\end{Lemma}

\begin{proof}
Observe that $S^{*} f = P^{+}( \overline{z} f)$, $f \in H^2$, and so 
\begin{align*}
X & = \overline{z} \overline{\K_{\alpha}} \; \; \mbox{for some $\alpha$}\\
& \iff \overline{z} \overline{X} = \K_{\alpha} \; \; \mbox{for some $\alpha$}\\\
%& \iff S^{*} (\overline{z} \overline{X}) \subset \overline{z} \overline{X} \; \; \mbox{and $\overline{z} \overline{X} \not = \overline{H^{2}}$}\\
& \iff P^{+} (\overline{z} (\overline{z} \overline{X})) \subset  \overline{z} \overline{X} \; \; \mbox{and $\overline{z} \overline{X} \not = H^{2}$}\\
& \iff z P^{+} (\overline{z} (\overline{z} \overline{X})) \subset  \overline{X} \; \;  \mbox{and $X \not = \overline{H^{2}_{0}}$.}
\end{align*}
Using the identity
$P_{-}(\overline{f}) = \overline{z P^{+}(\overline{z} f)},$
 we see that 
\begin{align*}
z P^{+} (\overline{z} (\overline{z} \overline{X})) \subset  \overline{X}& \iff P^{-}(z X) \subset X \; \; \mbox{and $X \not = \overline{H^{2}_0}$}  \qedhere 
\end{align*}
\end{proof}

\begin{Lemma}\label{77yttwwii9}
Let $u$ be any inner function and $\mathscr{S} \subset \K_{u}^{\perp}$ be a $D_u$-invariant subspace. If $P^{-} \mathscr{S} \not = \{0\}$ then there is an $f_{-} \in P^{-}\mathscr{S}$  such that $\phi_{f_{-}}(0) \not = 0$. 
\end{Lemma}

\begin{proof}
Suppose that for every $f_{-} \in P^{-} \mathscr{S} \setminus \{0\}$, with $f_{-} = P^{-} f$, $f \in \mathscr{S} \subset \K_{u}^{\perp}$, we have $\phi_{f}(0) = 0$. From Proposition \ref{0982iu4hrfjdslkva} and \eqref{ooopopoddGt666} we have 
$$P^{-}(D_{u} f) = P^{-}(z f_{-}) = z f_{-}$$ and so $z f_{-} \in  P^{-} \mathscr{S}$. Thus, by assumption, $z^2 f_{-} = \overline{\psi_{+}}$ with $\psi_{+} \in H^2$ and $\psi_{+}(0) = \phi_{f}'(0) = 0$. Therefore, 
$$P^{-}(D_{u}^{2} f) = P^{-}(z^2 f_{-}) = z^{2} f_{-} \in P^{-} \mathscr{S}.$$
Continuing in this manner we see that 
$$D_{u}^{n} f_{-} = z^n f_{-} \in \overline{H^{2}_{0}}, \quad n \geq 0,$$
which is impossible if $f_{-} \not = 0$. 
\end{proof}

These next two results further examine $P^{+} \mathscr{S}$ and $P^{-} \mathscr{S}$. 

\begin{Proposition}\label{9bgdlydndyd}
Let $u$ be any inner function and $\mathscr{S} \subset \K_{u}^{\perp}$ be a $D_u$-invariant subspace. Then one of the following three possibilities occurs: 
\begin{enumerate}
\item[(i)] $P^{-} \mathscr{S} = \{0\}$;
\item[(ii)] $P^{-} \mathscr{S} = \overline{H^{2}_{0}}$;
\item[(iii)] there is a non-constant inner function $\alpha$ such that $P^{-} \mathscr{S} = \overline{z} \overline{\K_{\alpha}}$.
\end{enumerate}
\end{Proposition}

\begin{proof}
Let $f_{-} \in P^{-} \mathscr{S}$. Then there is an $f  = f_{-} + u \widetilde{f_{+}} \in \mathscr{S}$. Thus 
$$D_{u} f = z f_{-} - \overline{\phi_{f}(0)} + u (z \widetilde{f_{+}} + \overline{u(0)} \overline{\phi_{f}(0)})$$
and so
$P^{-}(D_u f) = z f_{-} - \overline{\phi_{f}(0)} = P^{-}(z f_{-}).$
Since $D_u f \in \mathscr{S}$, we have $P^{-}(z f_{-}) \in P^{-} \mathscr{S}$. Apply Lemma \ref{888UyT} to $X = P^{-} \mathscr{S}$ to obtain the result. 
\end{proof}

\begin{Proposition}\label{3857654749uU}
Let $u$ be any inner function and $\mathscr{S} \subset \K_{u}^{\perp}$ be a $D_u$-invariant subspace. If $u \in P^{+} \mathscr{S}$, then $P^{+} \mathscr{S} = u H^2$. 
\end{Proposition}

\begin{proof}
Let $f \in \mathscr{S}$ with $f = f_{-} + u$. We have $P^{+}(D_u f) \in P^{+} \mathscr{S}$ and
\begin{align*}
D_{u} f& = D_{u}(f_{-} + u) = z f_{-} + z u - \overline{\phi_{f}(0)} + \overline{\phi_{f}(0)} \overline{u(0)} u\\
& = (z f_{-} - \overline{\phi_{f}(0)}) + z u + \overline{\phi_{f}(0)} \overline{u(0)} u.
\end{align*}
Thus $P^{+} D_{u} f = z u + \overline{\phi_{f}(0)} \overline{u(0)} u$ and so 
$z u = P^{+} D_{u} f - \overline{\phi_{f}(0)} \overline{u(0)} u \in P^{+} \mathscr{S}.$
Now let $f_{1} \in \mathscr{S}$ be such that $f_{1} = f_{1-} + u z$ with $f_{1-} \in \overline{H^{2}_{0}}$. Then 
$P^{+} D_{u} f_{1} = u z^2 - \overline{\phi_{f_{1}}(0)} \overline{u(0)} u,$
and it follows that $u z^2 \in P^{+} \mathscr{S}$. Analogously we conclude that $z^{j} u \in P^{+} \mathscr{S}$ for all $j \geq 0$ and so, since $P^{+} \mathscr{S} \subset u H^2$, we have $P^{+} \mathscr{S} = u H^2$.
%\begin{align*}
%D_u f & = D_{u} (f_{-} + u)\\
%& = z f_{-} + z u - (1 - \overline{u(0)} u) \int_{\T} (z f_{-} + z u) dm\\
%& = \Big(z f_{-}  - \int_{\T} z f_{-} dm
%\Big) + z u + \overline{u(0)}u  \int_{\T} z f_{-} dm.
%\end{align*}
%Thus 
%$$P^{+} (D_u f) = z u + \overline{u(0)} \int_{\T} z f_{-} dm$$ and so 
%$$z u = P^{+}(D_u f) - \overline{u(0)} u \int_{\T} z f_{-} dm \in P^{+} \mathscr{S}.$$
%Now let $f_1 \in \mathscr{S}$ be such that $f_{1} = f_{1-} + u z$ with $f_{1-} \in \overline{H^{2}_{0}}$. Then 
%$$P^{+}(D_u f_1) = u z^2 + u \overline{u(0)} \int_{\T} z f_{1-} dm.$$ and it follows that $u z^2 \in P^{+} \mathscr{S}$. Analogously we conclude that $z^{j} u \in P^{+} \mathscr{S}$ for all $j \geq 0$ and so, since $P^{+} \mathscr{S} \subset u H^2$, we have $P^{+} \mathscr{S} = u H^2$. 
\end{proof}

\section{Invariant subspaces when $u(0) \not = 0$}\label{unotzero}

 Theorem \ref{uuUUuuuv78778768768} says that when $u(0) \not = 0$,  $D_u$ is similar to $M$ on $L^2$. Results of Wiener and Helson \cite{MR0171178} together describe the $M$-invariant subspaces $\mathcal{F}$ of $L^2$ as follows: If $M \mathcal{F} = \mathcal{F}$, then there is a measurable subset $A \subset \T$ such that $\mathcal{F} = \chi_{A} L^2$ while if $M \mathcal{F} \not = \mathcal{F}$, then $\mathcal{F} = w H^2$ for some $w \in L^{\infty}$ with $|w| = 1$ almost everywhere on $\T$. This yields the following. 

\begin{Theorem}
Suppose $u$ is inner, $u(0) \not = 0$, and $\mathscr{S}$ is a $D_u$-invariant subspace of $\K_{u}^{\perp}$. When $D_{u} \mathscr{S} = \mathscr{S}$  there is a measurable $A \subset \T$ for which 
$$\mathscr{S} = (P^{-} + u \overline{u(0)} P^{+}) \chi_{A} L^2.$$
When $D_{u} \mathscr{S} \not = \mathscr{S}$, then 
$$\mathscr{S} =  (P^{-} + u \overline{u(0)} P^{+}) w H^2,$$
 for some $w \in L^{\infty}$ with $|w| = 1$ almost everywhere on $\T$.
\end{Theorem}

From $P^{-} + P^{+} = I$, we see that any $D_u$-invariant $\mathscr{S}$ takes the form 
$$\{g - k_{0}^{u} P^{+} g: g \in \mathcal{F}\},$$
where $\mathcal{F}$ is an $M$-invariant subspace of $L^2$.

Below are a few examples of
\begin{equation}\label{wwWwwW}
(P^{-} + u \overline{u(0)} P^{+})( wH^2)
\end{equation}
 for choices of inner  $u$ with $u(0) \not = 0$ and $w = \overline{\alpha} \beta$ for inner $\alpha$ and $\beta$. 

\begin{Example}
Let $u$ be inner with $u(0) \not = 0$. 
If $\alpha \equiv 1$ and $\beta$ is any inner function, then 
$$(P^{-} + u \overline{u(0)} P^{+})( \beta H^2) = u \beta H^2.$$
Observe how this connects to Proposition \ref{jjJjjllLll}.
\end{Example}

\begin{Example}
Let $u$ be inner with $u(0) \not = 0$. 
If $\beta \equiv 1$ and $\alpha$ is any inner function, then 
\begin{align*}
& (P^{-} + u \overline{u(0)} P^{+})(\overline{\alpha} H^2)\\
& = \{(P^{-} + u \overline{u(0)} P^{+})(\overline{\alpha} f_{+}): f_{+} \in H^2\}\\
& = \{(P^{-} + u \overline{u(0)} P^{+})(\overline{\alpha}(k + \alpha g_{+})): k \in \K_{\alpha}, g_{+} \in H^2\}\\
& = \{(P^{-} + u \overline{u(0)} P^{+})(\overline{\alpha} k + g_{+}): k \in \K_{\alpha}, g_{+} \in H^2\}.
\end{align*}
From Proposition \ref{978uwioejfskdlac} notice that for any $k \in \K_{\alpha}$ we have $\overline{\alpha} k \in \overline{H^{2}_{0}}$ and so 
$P^{-}(\overline{\alpha} k) = \overline{\alpha} k$ and $P^{+}(\overline{\alpha} k) = 0.$
Apply Proposition \ref{978uwioejfskdlac} to get 
$$
(P^{-} + u \overline{u(0)} P^{+})(\overline{\alpha} H^2)   = \overline{\alpha} \K_{\alpha} \oplus u H^2
 = \overline{z} \overline{\K_{\alpha}} \oplus u H^2.
$$
\end{Example}

\begin{Example}\label{0862387reiudf}
Let $\lambda \in \D \setminus \{0\}$ and 
$$u(z) = \alpha (z) = \frac{z - \lambda}{1 - \overline{\lambda} z}, \quad \beta(z) = z.$$
Then for any $f_{+} \in H^2$, 
\begin{align*}
& (P^{-} + u \overline{u(0)} P^{+})(\overline{\alpha} \beta f_{+})\\
 & = \big(P^{-} + \frac{z - \lambda}{1 - \overline{\lambda} z} (-\lambda) P^{+}\big)\big(\frac{1 - \overline{\lambda} z}{z - \lambda} z f_{+}\big)\\
& = P^{-}\Big(\frac{1 - \overline{\lambda} z}{z - \lambda} z f_{+}\Big) - \lambda \frac{z - \lambda}{1 - \overline{\lambda} z} P^{+}\Big(\frac{1 - \overline{\lambda} z}{z - \lambda} z f_{+}\Big)\\
& = \Big\{\frac{\overline{z}}{1 - \lambda \overline{z}} \lambda (1 - |\lambda|^2) f_{+}(\lambda) + \lambda \Big(z f_{+}  - \frac{\lambda (1 - |\lambda|^2)}{1 - \overline{\lambda} z} f_{+}(\lambda)\Big): f_{+} \in H^2\Big\}.
\end{align*}

%We now have 
%\begin{equation}\label{bBbbhh56tgsdhjfkgrf}
%\frac{1 - \overline{\lambda} z}{z - \lambda} z f_{+} =  \underbrace{\frac{(1 - \overline{\lambda} z) z f_{+} - (1 - |\lambda|^2) \lambda f_{+}(\lambda)}{z - \lambda}}_{\in H^2} + \underbrace{\frac{(1 - |\lambda|^2) \lambda f_{+}(\lambda)}{z - \lambda}}_{\in \overline{H^{2}_{0}}}.
%\end{equation}
%Thus 
%$$P^{-}\Big(\frac{1 - \overline{\lambda} z}{z - \lambda} z f_{+}\Big) = \frac{(1 - |\lambda|^2) \lambda f_{+}(\lambda)}{z - \lambda} = \frac{\overline{z}}{1 - \lambda \overline{z}} \lambda (1 - |\lambda|^2) f_{+}(\lambda).$$
%Since 
%$\K_{\alpha} = \C \frac{1}{1 - \overline{\lambda} z},$ we see from the previous line that 
%$$P^{-}\Big(\frac{1 - \overline{\lambda} z}{z - \lambda} z f_{+}\Big) \in \overline{z} \overline{\K_{\alpha}}.$$
%Moreover, from \eqref{bBbbhh56tgsdhjfkgrf} we see that 
%$$P^{+}\Big( \frac{1 - \overline{\lambda} z}{z - \lambda} z f_{+}\Big)= \frac{(1 - \overline{\lambda} z) z f_{+} - (1 - |\lambda|^2) \lambda f_{+}(\lambda)}{z - \lambda}.$$
%Putting this all together, we see from \eqref{333ee3EEE} that 
%$$(P^{-} + u \overline{u(0)} P^{+})(\overline{\alpha} \beta H^2)$$ is equal to 
%$$\Big\{\frac{\overline{z}}{1 - \lambda \overline{z}} \lambda (1 - |\lambda|^2) f_{+}(\lambda) + \lambda \Big(z f_{+}  - \frac{\lambda (1 - |\lambda|^2)}{1 - \overline{\lambda} z} f_{+}(\lambda)\Big): f_{+} \in H^2\Big\}.$$
The above is a proper subspace of $\overline{z} \overline{\K_{\alpha}} \oplus u H^2$. Indeed, $z - \lambda \in \overline{z} \overline{\K_{\alpha}} \oplus u H^2$ but there is no $f_{+} \in H^2$ for which 
$$z - \lambda = \frac{\overline{z}}{1 - \lambda \overline{z}} \lambda (1 - |\lambda|^2) f_{+}(\lambda) + \lambda \Big(z f_{+}  - \frac{\lambda (1 - |\lambda|^2)}{1 - \overline{\lambda} z} f_{+}(\lambda)\Big).$$ If there were such an $f_{+}$ then due to the uniqueness of orthogonal decomposition above then 
$f_{+}(\lambda) = 0$. This would mean that 
$z - \lambda = \lambda z f_{+}(z)$ for which there is no such $f_{+} \in H^2$.
\end{Example}

One can only go so far with these types of examples from \eqref{wwWwwW} since there are examples of unimodular $w$ which are not the quotient of two inner functions. 

%Our main result has a few interesting corollaries.

\begin{Corollary}\label{177654099}
Let $u$ be inner with $u(0) \not = 0$. If $\mathscr{S} \subset \K_{u}^{\perp}$, then $P^{-} V \mathscr{S} = P^{-} \mathscr{S}$. 
\end{Corollary}

\begin{proof}
If $g_{-} \in P^{-} V \mathscr{S}$ there is an $h \in \mathscr{S}$ such that 
\begin{align*}
g_{-} &= P^{-} (P^{-} + \frac{\overline{u}}{\overline{u(0)}} P^{+}) h
 = P^{-} h + P^{-} \frac{\overline{u}}{\overline{u(0)}} P^{+} h\\
& = P^{-} h + P^{-} \frac{\overline{u}}{\overline{u(0)}} \underbrace{(u h_1)}_{P^{+} h \in u H^2}
 = P^{-} h.
\end{align*}
Thus $P^{-} V \mathscr{S} \subset P^{-} \mathscr{S}$.

Conversely, if $h_{-} \in P^{-} \mathscr{S}$, there exists an $h \in \mathscr{S}$ such that $h_{-} = P^{-} h$. Thus for 
$$g = \Big(P^{-} + \frac{\overline{u}}{\overline{u(0)}} P^{+}\Big) h \in V \mathscr{S},$$
we have $P^{-} g = h_{-}$. Thus $P^{-} \mathscr{S} \subset P^{-} V \mathscr{S}$. 
\end{proof}

\begin{Corollary}
Let $u$ be inner with $u(0) \not = 0$. If $\mathscr{S} \subset \K_{u}^{\perp}$ is a $D_u$-invariant subspace and $\{0\} \subsetneq P^{-} \mathscr{S} \subsetneq \overline{H^{2}_{0}}$, then $\mathscr{S} = V^{-1} (\overline{\alpha} \beta H^2)$ for two coprime inner functions $\alpha$ and $\beta$. 
\end{Corollary}

\begin{proof}
By Proposition \ref{9bgdlydndyd} we have $P^{-} \mathscr{S} = \overline{z} \overline{\K_{\alpha}}$ for some inner function $\alpha$ and by Corollary \ref{177654099}, $P^{-} V \mathscr{S} = P^{-} \mathscr{S} = \overline{\alpha} \K_{\alpha}$. Thus 
\begin{align*}
V \mathscr{S}  = (P^{-} + P^{+}) V \mathscr{S}
 \subset P^{-} V \mathscr{S} \oplus P^{+} V \mathscr{S}
 = \overline{\alpha} K_{\alpha} \oplus P^{+} V \mathscr{S}.
\end{align*}
Thus $\alpha V \mathscr{S} \subset \K_{\alpha} \oplus \alpha P^{+} V \mathscr{S} \subset H^2$. By Theorem  \ref{uuUUuuuv78778768768}, $\alpha V \mathscr{S}$ is an $S$-invariant subspace of $H^2$ which means that $\alpha V \mathscr{S} = \beta H^2$ for some inner function $\beta$. Dividing out by any common inner factors between $\alpha$ and $\beta$ we can assume that $\alpha$ and $\beta$ are coprime. Thus $\mathscr{S} = V^{-1}(\overline{\alpha} \beta H^2)$.
\end{proof}

%A contradiction argument with the previous corollary yields the following. 

\begin{Corollary}
Let $\mathcal{F}$ be an $M$-invariant subspace of $L^2$ that is not of the form $\overline{\alpha} \beta H^2$ for inner $\alpha$ and $\beta$. Then $\mathscr{S} = V^{-1} \mathcal{F}$ is a $D_u$-invariant subspace with $P^{-} \mathscr{S} = \overline{H^{2}_{0}}$.
\end{Corollary}

\begin{Remark}
\hfill 
\begin{enumerate}
\item[(i)] The theorems in this section identify $P^{-} \mathscr{S}$ and $P^{+}\mathscr{S}$ separately. It is interesting that $\mathscr{S}$ can be a proper subset of $P^{-} \mathscr{S} \oplus P^{+} \mathscr{S}$ which seems to create a rich invariant subspace structure.
\item[(ii)] If $u(0) \not = 0$ and $\mathscr{S} \not = \{0\}$, we do not have $P^{+} \mathscr{S} = \{0\}$. Indeed, this would mean that $\mathscr{S} \subset \overline{H^{2}_{0}}$. However, for any $f_{-} \in \mathscr{S}$ we would have 
$$D_{u} f_{-} = z f_{-} - \overline{\phi_{f}(0)} + \overline{u(0)} \overline{\phi_{f}(0)} u \not \in \overline{H^{2}_{0}}$$
if $\phi_{f}(0) \not = 0$ (Lemma \ref{77yttwwii9}). 
\end{enumerate}
\end{Remark}

\begin{comment}
\begin{Remark}
So, thought one might be tempted to think when $u(0) \not = 0$, that every $D_{u}$-invariant subspace of $\K_{u}^{\perp}$ is $\overline{z} \overline{\K_{\alpha}} \oplus u H^2$ for some inner $\alpha$, this is not the case. 
\end{Remark}
\end{comment}

\section{Invariant subspaces when $u(0) = 0$}\label{uzero}

We characterized the $D_u$-invariant subspaces of $\K_{u}^{\perp}$ when $u(0) \not = 0$. We now discuss the $u(0) = 0$ case. 

\begin{comment}
From Proposition \ref{jjJjjllLll} and Proposition \ref{kkKkkKkKkkKK}, we know that when if $\mathcal{S} \subset \K_{u}^{\perp}$ is $D_u$-invariant and $P^{-} \mathcal{S} = \{0\}$, then $\mathcal{S} = \gamma u H^2$ for some inner function $\beta$. We also know that when $P^{-} \mathcal{S} \not = \{0\}$, then $P^{-} \mathcal{S} = \overline{z} \overline{\K_{\alpha}}$. We also have the following collection of $D_u$-invariant subspaces. 
\end{comment}

\begin{Proposition}\label{pPPoOOooo0098}
Let $u$ be inner with $u(0) = 0$. If $\alpha$ and $\gamma$ are inner then  $\overline{z} \overline{\K_{\alpha}} \oplus \gamma u H^2$ is a $D_u$-invariant subspace of $\K_{u}^{\perp}$.
\end{Proposition}

\begin{proof}
Let $f = \overline{z} \overline{k} + \gamma u h$, where $k \in \K_{\alpha}$, $h \in H^2$. Proposition \ref{0982iu4hrfjdslkva} yields
\begin{align*}
D_{u}(\overline{z} \overline{k} + \gamma u h) & 
 = (\overline{k} - \overline{k(0)}) + z \gamma u h\\
& = \overline{z} \cdot \frac{\overline{k}  - \overline{k(0)}}{\overline{z}} + z \gamma u h \in \overline{z} \overline{\K_{\alpha}} + \gamma u H^2,
\end{align*}
where we took into account that $k \in \K_{\alpha} \implies \overline{z} (k - k(0)) \in \K_{\alpha}$. 
\end{proof}

%One can also see the above result by noting that $D_u$ is unitarily equivalent to $S \oplus S^{*}$ via Corollary \ref{bb77ujmkjhnb}. Indeed, consider  the $S \oplus S^{*}$-invariant subspace $\gamma H^2 \oplus \K_{\alpha}$ for inner $\gamma$ and $\alpha$ and follow the unitary operators from the proof of Corollary \ref{bb77ujmkjhnb} to produce the $D_u$-invariant subspace $\overline{z} \overline{\K_{\alpha}} \oplus \gamma u H^2$.

\begin{Proposition}\label{54433212}
Suppose $u$ is inner with $u(0) = 0$ and $\mathscr{S} \subset \K_{u}^{\perp}$ is $D_u$-invariant. Then either $P^{+} \mathscr{S} = \{0\}$ or $P^{+} \mathscr{S} = \gamma u H^2$ where $\gamma$ is inner.
\end{Proposition}

\begin{proof}
Let $P^{+} \mathscr{S} \not = \{0\}$ and $f = f_{-} + u \widetilde{f_{+}} \in \mathscr{S}$. Then 
$$P^{+}(D_u f) = u (z \widetilde{f_{+}} + \overline{u(0)} \overline{\phi_{f}(0)}) = z u \widetilde{f_{+}} \in P^{+} \mathscr{S}.$$
Thus $P^{+} \mathscr{S}$ (which is a subspace of $u H^2$) is a non-zero $S$-invariant subspace and thus, by Beurling's Theorem, $P^{+} \mathscr{S} = \gamma u H^2$ for some inner $\gamma$. 
\end{proof}

Proposition \ref{pPPoOOooo0098} does not describe all the $D_u$-invariant subspaces of $\K_{u}^{\perp}$. To get a better understanding where the complication lies, and since this is an interesting problem in its own right, let us cast this in an equivalent setting. From Corollary \ref{bb77ujmkjhnb}, a description of the $D_u$-invariant subspaces of $\K_{u}^{\perp}$ will yield a description of the $S \oplus S^{*}$-invariant subspaces of $H^2 \oplus H^2$. One can also check that the unitary operator that makes these two operators equivalent takes the $D_u$-invariant subspace $\gamma u H^2 \oplus \overline{z} \overline{\K_{\alpha}}$ to the $S \oplus S^{*}$-invariant subspace $\gamma u H^2 \oplus \K_{\alpha}$. However, these are not all of them. 

\begin{Example}\label{1072886tT}
For $a \in \D \setminus \{0\}$  consider the $S \oplus S^{*}$-invariant subspace generated by 
$$\frac{1}{1 - \overline{a} z} \oplus \frac{1}{1 - \overline{a} z},$$
that is, 
$$\bigvee\Big\{(S \oplus S^{*})^n \big(\frac{1}{1 - \overline{a} z} \oplus \frac{1}{1 - \overline{a} z}\big): n \geq 0\Big\}.$$
For any polynomial $p(z)$ we have 
$$p(S \oplus S^{*})\big(\frac{1}{1 - \overline{a} z} \oplus \frac{1}{1 - \overline{a} z}\big) = \frac{p(z)}{1 - \overline{a} z} \oplus \frac{p(\overline{a})}{1 - \overline{a} z}.$$
If $\{p_n\}_{n \geq 1}$ is a sequence of polynomials with 
$$p_n(S \oplus S^{*})\big(\frac{1}{1 - \overline{a} z} \oplus \frac{1}{1 - \overline{a} z}\big) \to f \oplus g$$ in $H^2 \oplus H^2$,  one can argue that 
$p_n(z) \to (1 - \overline{a} z) f$ in the norm of $H^2$ and thus
$p_{n}(\overline{a}) \to (1 - \overline{a}^2) f(\overline{a}).$
Thus
$$\bigvee\Big\{(S \oplus S^{*})^n \big(\frac{1}{1 - \overline{a} z} \oplus \frac{1}{1 - \overline{a} z}\big): n \geq 0\Big\} = \Big\{f \oplus \frac{f(\overline{a}) (1 - \overline{a}^2)}{1 - \overline{a} z}: f \in H^2\Big\}.$$ 
This subspace is contained in $H^2 \oplus K_{\alpha}$, where 
$$\alpha(z) = \frac{z - a}{1 - \overline{a} z},$$ but the containment is proper. Indeed, we have 
$$1 \oplus \frac{1}{1 - \overline{a} z} \in H^2 \oplus \K_{\alpha}.$$ However, 
$$1 \oplus \frac{1}{1 - \overline{a} z} \not \in \Big\{f \oplus \frac{f(\overline{a}) (1 - \overline{a}^2)}{1 - \overline{a} z}: f \in H^2\Big\}.$$ 
\end{Example}

This leads to the question: What are the invariant subspaces of $S \oplus S^{*}$?

%The key step in describing the $D_u$-invariant subspaces when $u(0) \not = 0$ was showing that $D_u$ is similar to $M_z$ on $L^2$. There can be no such result when $u(0) = 0$. To see this, note that the spectrum of $M_z$ on $L^2$ is $\T$ while the spectrum of $S \oplus S^{*}$ is $\overline{\D}$. Note that 
%$$0 \oplus \frac{1}{1 - \overline{a} z}, \quad a \in \D$$ are eigenvectors for $S \oplus S^{*}$ with eigenvalues $\overline{a}$. This also proves the first part of Proposition \ref{yiusyfisudfsa111}.

\section{Orthogonal sums}

A complicating factor is that for a $D_u$-invariant subspace $\mathscr{S}$ we may not have $P^{\pm} \mathscr{S} \subset \mathscr{S}$. We always have 
$\mathscr{S} \subset P^{-} \mathscr{S} \oplus P^{+} \mathscr{S}$ but Example \ref{0862387reiudf} shows this containment can be proper. Our main theorem  is the following.

\begin{Theorem}
Let $u$ be a inner and  $\mathscr{S}$ be a non-trivial $D_u$-invariant subspace of the form 
$\mathscr{S} = X_{-} \oplus Y_{+},$ where $X_{-}$ is a closed subspace of $\overline{H^{2}_{0}}$ and $Y_{+}$ is a closed subspace of $u H^2$.  
\begin{enumerate}
\item[(i)] If $u(0) \not = 0$, then $\mathscr{S}$ takes one of the forms: $\gamma u H^2$ or $\overline{z \K_{\alpha}} \oplus u H^2$, where $\gamma$ and $\alpha$ are inner. 
\item[(ii)] If $u(0) = 0$, then $\mathscr{S}$ takes  one of the following forms: $\overline{H^{2}_{0}}$, $\overline{z \K_{\alpha}}$, $\gamma u H^2$, $\overline{H^{2}_{0}} \oplus \gamma u H^2$,  or $\overline{z \K_{\alpha}} \oplus \gamma u H^2$, where $\gamma$ and $\alpha$ are inner. 
\end{enumerate}
\end{Theorem}
 
 \begin{proof}
 Proof of $(i)$. By Proposition \ref{jjJjjllLll} we see that if $X_{-} = \{0\}$, then $Y_{+} = \gamma u H^2$. On the other hand, if $X_{-} \not = \{0\}$ then by Lemma \ref{77yttwwii9}  there is an $f_{-} \in X_{-} \subset \mathscr{S}$  such that for
$\phi_{f}  = \overline{z} \overline{f_{-}}$ we have $\phi_{f}(0) \not = 0$. Furthermore, 
$D_{u} f_{-} = z f_{-} - \overline{\phi_{f}(0)} + \overline{u(0)} \overline{\phi_{f}(0)} u \in \mathscr{S}.$
Therefore, 
$$P^{-}(D_{u} f_{-}) = z f_{-} - \overline{\phi_{f}(0)} \in X_{-} \subset \mathscr{S}.$$
These equations imply that $u \in \mathscr{S}$. Proposition \ref{3857654749uU} implies $Y_{+} = u H^2$. Proposition \ref{9bgdlydndyd} says that either $X_{-}  = \overline{H^{2}_{0}}$, which yields 
$$\mathscr{S} = X_{-} \oplus Y_{+} = \overline{H^{2}_{0}} \oplus u H^2 = \K_{u}^{\perp}$$  or $X_{-} = \overline{z \K_{\alpha}}$, which implies $\mathscr{S} = \overline{z \K_{\alpha}} \oplus u H^2$. 

Proof of $(ii)$. Proposition \ref{54433212} says that either $Y_{+} = \{0\}$ or $Y_{+} = \gamma u H^2$ for some inner $\gamma$. Thus $\mathscr{S} = X_{-}$ or $\mathscr{S} = X_{-} \oplus \gamma u H^2$. Proposition \ref{9bgdlydndyd}  says that either $X_{-} = \{0\}$, $X_{-} = \overline{H^{2}_{0}}$, or $X_{-} = \overline{z \K_{\alpha}}$. 
 \end{proof}

\bibliographystyle{plain}

\bibliography{references}

\end{document}